\documentclass[12pt,a4paper]{amsart}
\usepackage{graphics}
\usepackage{epsfig}
\usepackage{graphicx}
\theoremstyle{plain}
\usepackage{amssymb}
\usepackage{color}
\advance\hoffset-20mm \advance\textwidth40mm

\newtheorem{theorem}{Theorem}
\newtheorem{lemma}{Lemma}
\newtheorem*{theo*}{Theorem}
\newtheorem{proposition}{Proposition}
\newtheorem{corollary}{Corollary}
\theoremstyle{definition}
\newtheorem{definition}{Definition}
\newtheorem*{definition*}{Definition}
\newtheorem{example}{Example}
\newtheorem{remark}{Remark}

\begin{document}
\sloppy
\title[Maximal subalgebras of the Lie algebra $W_n(\mathbb{K})$]
{Maximal subalgebras of the Lie algebra $W_n(\mathbb{K})$}
\author
{Y.Chapovskyi, A.Petravchuk, O.Tyshchenko}
\address{Institute of Mathematics, National Academy of Sciences of Ukraine,
Tereschenkivska street, 3, 01004 Kyiv, Ukraine}
\email{safemacc@gmail.com}
\address{ Faculty of Mechanics and Mathematics,
Taras Shevchenko National University of Kyiv, 64, Volodymyrska street, 01033  Kyiv, Ukraine}
\email{ petravchuk@knu.ua, apetrav@gmail.com}
\address{ Faculty of Mechanics and Mathematics,
	Taras Shevchenko National University of Kyiv, 64, Volodymyrska street, 01033  Kyiv, Ukraine}
\email{oleg.tyshchenko@knu.ua}

\date{\today}

\begin{abstract}
Let $K$ be an algebraically closed field of characteristic zero, $A= K[x_1, \dots, x_n]$ the polynomial ring in $n$ variables, and let $W_n(K)$ be the Lie algebra of all $K$-derivations of $A.$  This Lie algebra also  is  the free $A$-module of rank $n$ over the ring $A,$ so every subalgebra of $W_n(K)$ has a rank $\leq n$ over $A.$ We prove that every maximal subalgebra of rank $\leq n$ of $W_n(K)$ is a simple Lie algebra. If a maximal subalgebra $L\subset W_n(K)$ has rank $n$ and is a submodule of $W_n(K)$ then $L$ is not simple. Moreover, $L$ is of the form $L=\{ D\in W_n(K) \  |  \ D(I)\subseteq I\} $ for some ideal $I$ of the ring $A.$ It is also proved that, for a simple derivation $D$ on the ring $K[x, y]$, the subalgebra $K[x, y]D$ is a maximal subalgebra   of $W_2(K).$ 
 \end{abstract}
\maketitle

%%%%%%%%%%%%%%%%%%%%%%%%%%%%%%%%%%%%%%%%%%%%%%%%%%%%%%%%%%%
	\section{Introduction}
Let $K$ be an algebraically closed field of characteristic zero, $A = K[x_1, \dots, x_n]$ the polynomial ring, and $R = K(x_1, \dots, x_n)$ the field of rational functions.
Recall that a $K$-linear map $D: A \rightarrow A$ is called a $K$-derivation on $A$ if it satisfies the Leibniz's rule:	$D(fg) = D(f)g + fD(g)$  for all $ f, g \in A.$
Every $K$-derivation (or simply derivation) can be uniquely written in the form:
$$	D = f_1 \frac{\partial}{\partial x_1} + \dots + f_n \frac{\partial}{\partial x_n}, \ \ 
f_i=f_(x_1, \ldots , x_n) \in A,$$ where $\frac{\partial}{\partial x_i}$ are partial derivatives.
The vector space  $W_n(K)$ (over the field $K$) of all derivations on $A$ is a Lie algebra over $K$ with respect to the Lie bracket $	[D_1, D_2] = D_1 \circ D_2 - D_2 \circ D_1;$  we denote this Lie algebra  by $W_n.$ 
On the other hand,  $W_n$ is a free module over $A$ with the standard basis $\frac{\partial}{\partial x_1}, \dots, \frac{\partial}{\partial x_n}$. The Lie algebra $W_n$ is of great interest from the  geometric point of view, since  $W_n$ consists of all polynomial vector fields on $K^n$.
Although the Lie algebra $W_n$ has been  studied by many authors, the structure of its subalgebras is still  poorly understood. In particular,  little is known about its maximal subalgebras. Some examples of maximal subalgebras can be found in \cite{Rud1}, \cite{Bavula}, \cite{ESS}. In the case $n=1$, that is, for $W_1$,  a description of maximal subalgebras of $W_1$  is given  in \cite{Bell}.
Maximal subalgebras of $W_n(K)$ are of interest because every proper subalgebra of $W_n(K)$ is contained in a maximal subalgebra.  This follows from the fact  that $W_n(K)$ is finitely generated as a Lie algebra. 	 

In this paper, we study several classes of maximal subalgebras of the Lie algebra $W_n(K)$. We use  results from  \cite{J1} concerning the relationship  between ideals of subalgebras of $W_n$ and invariant ideals of the polynomial ring $K[x_1, \dots, x_n].$ 
As $W_n(K)$ is the free module over the ring $A$, for any subalgebra $L \subseteq W_n$, one can define its rank   ${\rm rk}_A L$ by the rule
${\rm rk}_A L = \dim_{R} RL.$ It is easy to show that, for any subalgebra $L \subseteq W_n$, the $A$-span  $AL$ is also a subalgebra of $W_n(K)$ and a submodule of the free $A$-module $W_n.$ Such subalgebras will be called \textit{polynomial} subalgebras.

The main results obtained in this paper are the following:

(1) Every nonzero maximal subalgebra $L$ of rank $<n$ of $W_n(K)$ is a simple Lie algebra ({Theorem \ref{th1}).
	
	(2)		Every polynomial maximal  subalgebra $M$ of rank $n$ in $W_n$ is of the form $M=\{ D \in W_n \mid D(I) \subseteq I \}$ for some proper ideal $I$ of the polynomial ring $A.$ In particular,   $M$ is a not a simple Lie algebra. (Theorem \ref{th3a}).		
	
	(3) If $D$ is a simple derivation of the polynomial ring $K[x, y]$, then $AD$ is a maximal subalgebra of $W_2(K).$ Conversely, every maximal subalgebra of $W_2$ of rank 1 is of the form $AD$ for a simple derivation $D$ on $A$ (Theorem \ref{th2}).
	
	(4) Every finite dimensional maximal subalgebra of $W_n$ has dimension  $\ge n+1$ (Theorem \ref{th3}). Moreover, there exists a maximal subalgebra of dimension $n^2 + 2n$ of $W_n(K),$  this subalgebra is isomorphic to the simple Lie algebra $\mathfrak{sl}_{n+1}(K)$ (Theorem \ref{th5}).
	
	We use standard notation in the paper.  Recall that $W_n(\mathbb{K})$ has the standard grading
	${W_n = \bigoplus_{i \ge -1} W_n^{[i]}}$, where
	$$W_n^{[i]} = \{ \sum_{j=1}^n f_j \frac{\partial}{\partial x_j} \mid   \text{where} \ f_j  \ \text{is homogeneous}  \ \text{with} \   \deg f_j = i+1, \ j=1,\ldots n  \}.$$
	The component $W_n^{[0]}$ is a subalgebra of $W_n(K)$ and $W_n^{[0]} \simeq \mathfrak{gl}_n({K})$. Besides $W_n^{[0]} = M_0 \oplus N_0$ where $M_0$ consists of linear derivations with zero divergence, $N_0 = KE_n$ where $E_n = \sum_{i=1}^n x_i \frac{\partial}{\partial x_i}$ is the Euler derivation. Since $M_0 \simeq \mathfrak{sl}_n(K)$, we see that all the components $W_n^{[i]}$ are 
	$\mathfrak{sl}_n(K)$-modules of finite dimension over ${K}$. A derivation $D\in W_n $	is called a simple derivation if the only $D$-invariant ideals of the polynomial ring $A=K[x_1, \ldots , x_n]$ are $\{ 0\}$ and $A.$
	If $D = \sum _{i=1}^nf_i \frac{\partial}{\partial x_i}$, then the divergence $\text{div} D$ is defined in the usual way: $\operatorname{div} D = \sum_{i=1}^n \frac{\partial f_i}{\partial x_i}$. 
	A derivation $D\in W_2$ is called a Jacobian derivation if there exists a polynomial $f\in K[x, y]$ such that, for every $h\in K[x, y],$ we have $D(h)=\det J(f, h)$, where $J(f, h)$ is the Jacobi matrix of the polynomials $f, h.$ Such a derivation will be denoted by  $D_f.$  
	
	\section{ Maximal subalgebras of  rank $<n$}
	\begin{lemma}
		Let $X = V(I), \  X \subset K^n$ be an affine algebraic variety with the defining ideal $I$, and let $p\in K^n\setminus X$ be a point. Then, for an arbitrary vector $\vec{u} = (u_1, \dots, u_n) \in K^n$, there exists a derivation $D = \sum_{i=1}^n f_i(x_1, \dots, x_n) \frac{\partial}{\partial x_i}\in W_n$  such that:
		\begin{enumerate}
			\item $D(I) \subseteq I$;
			\item $v_D(p) = (u_1, \dots, u_n)$, where $v_D$ is the vector field on $K^n$ corresponding to $D$.
		\end{enumerate}
	\end{lemma}
	
	\begin{proof}
		(1) As $p \notin X$ there exists a polynomial $F(x_1, \dots, x_n) \in I$ such that $F(p) \neq 0$. Consider the derivation
		\[
		D = \frac{F(x_1, \dots, x_n)}{F(p)} \left( u_1 \frac{\partial}{\partial x_1} + \dots + u_n \frac{\partial}{\partial x_n} \right).
		\]
		It is obvious that $D(I) \subseteq I$ (since $F(x_1, \dots, x_n) \in I$), so the part (1) is proved.
		
		The part (2) of the lemma is obvious.
	\end{proof}	
	\begin{theorem}\label{th1}
		Let $M$ be a maximal subalgebra of the Lie algebra $W_n$ with $\text{rk}M = k < n$. Then
		
		{\rm (1)} the only $M$-invariant ideals of the polynomial ring $A$ are $A$ and $0.$
		
		{\rm (2)} $M$ is a simple subalgebra of the Lie algebra $W_n$.
	\end{theorem}
	
	\begin{proof}
		(1) Assume,  to the contrary, that there exists a nonzero ideal $I$ of the ring $A$, $I \neq A$ such that $M(I) \subseteq I$. Denote by $M_1$ the Lie subalgebra of the form
		\[
		M_1 = \{ D \in W_n \mid D(I) \subseteq I \}.
		\]
		Obviously $M \subseteq M_1$. Note that $M_1 \neq W_n$ because there exist no proper nonzero ideals $I$ such that $W_n(I) \subseteq I$. It follows from the maximality of the subalgebra $M$ that $M = M_1$. Fix any point $p \in K^n \setminus V(I)$, where $V(I)$ is the affine algebraic variety with the defining ideal $I$.
		By Lemma 1, for any vector $\vec{u} = (u_1, \dots, u_n) \in K^n$ there exists a derivation $D \in M_1$, such that $v_D(p)=\vec{u}.  $ The latter means that 
		$\{ v_D(p) \mid D \in M_1 \} = K^n$. On the other hand, $M$ is of rank $k < n$ over $A$, so the vector space $\{ v_D(p) \mid D \in M \}$ is obviously of dimension $\le k$.		
		The obtained contradiction proves the  part (1) of the theorem.
		
		(2) Since the subalgebra $M$ admits  no proper non-zero invariant ideals of the polynomial ring $A$ we have by \cite{J1},  Theorem 3, that  $M$ is a simple Lie algebra.
	\end{proof}

	\begin{lemma}\label{lemma1}
		Let $L \subseteq W_n(K)$ be a nonzero subalgebra and let $D_1, D_2 \in L$ be such that $AD_1+AD_2 \subseteq L$. Denote by $S$ the subalgebra of $W_n$ generated by $D_1, D_2$, i.e., $S = \langle D_1, D_2 \rangle$. Then $AS \subseteq L$.
	\end{lemma}
	
	\begin{proof}
		One can easily check that $S = \bigcup_{i=1}^\infty S_i$ where $S_i$ can be built recursively
		$S_1 = \text{span} \{ D_1, D_2 \}$, $S_{i+1} = S_i + [S_i, S_i]$.
		Let us show by induction on $i$ that $AS_i \subseteq L$ for $i \geq 1$. Indeed, $AS_1 = AD_1+AD_2 \subseteq L$ by the conditions of the lemma. Assume that $AS_{i-1} \subseteq L$ and show that $AS_i \subseteq L$. Since $S_i = S_{i-1} + [S_{i-1}, S_{i-1}]$ it is sufficient to prove that $A[S_{i-1}, S_{i-1}] \subseteq L$. Take any elements $T_1, T_2 \in S_{i-1}$. Then, for any polynomial $f \in A$, we have the equality
		\[ f [T_1, T_2] = [f T_1, T_2] + T_2(f) T_1 \]
		By the inductive assumption $T_2(f) T_1 \in AS_{i-1}\subseteq L$ and $[f T_1, T_2] \in [AS_{i-1}, S_{i-1}] \subseteq L$. The latter means that  $f [T_1, T_2] \in L$. It follows from this inclusion that $A[S_{i-1}, S_{i-1}] \subseteq L.$ Therefore $AS_i \subseteq L$ and $AS\subseteq L.$ 
		The proof is complete.
	\end{proof}
	
	\begin{lemma}\label{lemma2}
		Let $D$ be a simple derivation of the polynomial ring $A=K[x_1, \dots, x_n]$. Then $AD$ is a maximal subalgebra in the partially ordered set (relative to inclusion) of all subalgebras of rank 1 in $W_n(K)$.
	\end{lemma}
	
	\begin{proof} Write $D$ in the coordinate form
		$$	
		D = \sum_{i=1}^{n}a_i \, \frac{\partial}{\partial x_i}, \quad a_i=a_i(x_1, \ldots , x_n)\in A.
		$$	
		Note that the ideal $I=(a_1, \ldots, a_n)$ of the ring $A$ is $D$-invariant, i.e., $
		D(I) \subseteq I.$
		Since $D$ is simple and $D \neq 0$ we have that $I=A$. The latter means that
		$\gcd(a_1, \dots, a_n) = 1. $ 		Assume, to the contrary, that $AD$ is strictly contained in a subalgebra $M$ of  $\rm{rank} = 1$ from $W_n$. Take any $T \in M \setminus AD$, and write
		\[
		T = \sum_{i=1}^{n} b_i \,  \frac{\partial}{\partial x_i}, \quad b_i \in A.
		\]
		Since $\rm{rank}(M) = 1$ there exist $c, d \in A$ such that
		$cD + dT = 0.$ 		Without loss of generality one can assume that $\gcd(c,d) = 1$. But then $ca_i = db_i, \ i = 1, \dots,n$. Since $\gcd(a_1,\dots,a_n)=1$
		we get $d \in \mathbb{K}^*$. It follows from the last equality that $b_i = d^{-1}ca_i, \ i = 1,\dots,n$, i.e., $T = d^{-1}cD \in AD$. The latter contradicts our choice of $T$ and therefore $AD$ is maximal among subalgebras of $\rm{rank}$ 1 in $W_n(K)$. The proof is complete.
	\end{proof}
	The next statement seems to be known, but having no precise reference we supply it with a proof.	
	\begin{lemma}\label{invariant}
		Let $D_1, D_2 \in W_n(\mathbb{K})$ be  linearly independent over $A$ derivations such that    
		$[D_1, D_2] = \mu_1 D_1 + \mu_2 D_2$ for some polynomials $\mu_1, \mu_2 \in A.$ Denote 
		$\Delta_{ij} = \begin{vmatrix} D_1(x_i) & D_1(x_j) \\ D_2(x_i) & D_2(x_j) \end{vmatrix}$.
		Then the next equalities hold 
		$$
		D_1(\Delta_{ij}) = \mu_2 \Delta_{ij} + \sum_{k=1}^n \left(\frac{\partial}{\partial x_k} D_1(x_i)\right) \Delta_{kj} - \sum_{k=1}^n \left(\frac{\partial}{\partial x_k} D_1(x_j)\right) \Delta_{ki},	$$
		$$ 	D_2(\Delta_{ij}) = -\mu_1 \Delta_{ij} + \sum_{k=1}^n \left( \frac{\partial}{\partial x_k} D_2(x_i) \right) \Delta_{kj} - \sum_{k=1}^n \left( \frac{\partial}{\partial x_k} D_2(x_j) \right) \Delta_{ki}. 	$$
		Therefore  the ideal $I$ of the polynomial ring $A$ generated by the determinants $\Delta_{ij}$, $i,j=1,\dots,n$ is invariant under action of the derivations $D_1, D_2$.
	\end{lemma}
	
	\begin{proof} We prove only the first equality, the proof of the second is similar. We obviously have
		$$D_1(\Delta_{ij}) = \begin{vmatrix} D_1^2(x_i) & D_1^2(x_j) \\ D_2(x_i) & D_2(x_j) \end{vmatrix} + \begin{vmatrix} D_1(x_i) & D_1(x_j) \\ D_1 D_2(x_i) & D_1 D_2(x_j) \end{vmatrix}=$$
		
		$$= \begin{vmatrix} D_1^2(x_i) & D_1^2(x_j) \\ D_2(x_i) & D_2(x_j) \end{vmatrix} + \begin{vmatrix} D_1(x_i) & D_1(x_j) \\ [D_1, D_2](x_i) & [D_1, D_2](x_j) \end{vmatrix} + \begin{vmatrix} D_1(x_i) & D_1(x_j) \\ D_2 D_1(x_i) & D_2 D_1(x_j) \end{vmatrix}.$$
		
		Taking into account the equality $[D_1, D_2] = \mu_1 D_1 + \mu_2 D_2  $ we obtain the expression for $D_1(\Delta_{ij})$ of the form
		
		$$ \begin{vmatrix} D_1^2(x_i) & D_1^2(x_j) \\ D_2(x_i) & D_2(x_j) \end{vmatrix} + \begin{vmatrix} D_1(x_i) & D_1(x_j) \\ \mu_1 D_1(x_i) + \mu_2 D_2(x_i) & \mu_1 D_1(x_j) + \mu_2 D_2(x_j) \end{vmatrix} + \begin{vmatrix} D_1(x_i) & D_1(x_j) \\ D_2 D_1(x_i) & D_2 D_1(x_j) \end{vmatrix}.$$
		Therefore 
		$$D_1(\Delta_{ij})=\mu_2 \Delta_{ij} + D_1^2(x_i) D_2(x_j) - D_1^2(x_j) D_2(x_i) + (D_2 D_1(x_j)) D_1(x_i) - (D_2 D_1(x_i)) D_1(x_j).$$
		In view of the expression $D=\sum _{k=1}^nD(x_k)\frac{\partial}{\partial x_k}$ for every $D\in W_n,$ we can write the right side of the last equality in the form
		$$ \mu_2 \Delta_{ij} + \left( \sum_{k=1}^n \left( \frac{\partial}{\partial x_k} D_1(x_i) \right) D_1(x_k) \right) D_2(x_j) - \left( \sum_{k=1}^n \left( \frac{\partial}{\partial x_k} D_1(x_j) \right) D_1(x_k) \right) D_2(x_i)+$$
		$$\quad + \left( \sum_{k=1}^n \left( \frac{\partial}{\partial x_k} D_1(x_j) \right) D_2(x_k) \right) D_1(x_i) - \left( \sum_{k=1}^n \left( \frac{\partial}{\partial x_k} D_1(x_i) \right) D_2(x_k) \right) D_1(x_j).$$
		After straightforward simplifications, we obtain the desired equality	
		$$D_1(\Delta_{ij})= \mu_2 \Delta_{ij} + \sum_{k=1}^n \left( \frac{\partial}{\partial x_k} D_1(x_i) \right) \Delta_{kj} - \sum_{k=1}^n \left( \frac{\partial}{\partial x_k} D_1(x_j) \right) \Delta_{ki}.$$
	\end{proof}
	
	\begin{corollary}\label{delta}
		Let $D_1, D_2 \in W_2(\mathbb{K})$ be linearly independent over~$\mathbb{K}[x_1,x_2]$ derivations such that  
		$[D_1, D_2] = \mu_1 D_1 + \mu_2 D_2$ for some polynomials $\mu_1, \mu_2\in \mathbb{K}[x_1,x_2].$ Denote  $\Delta_{12} = \begin{vmatrix} D_1(x_1) & D_1(x_2) \\ D_2(x_1) & D_2(x_2) \end{vmatrix}$.
		Then $
		D_1(\Delta_{12}) = (\mu_2 + {\rm div}D_1) \Delta_{12} ,	$
		$	D_2(\Delta_{12}) = (-\mu_1 + {\rm div}D_2) \Delta_{12}.	$
		Thus, the determinant~$\Delta_{12}$ is a common Darboux polynomial of the derivations $D_1, D_2$.
	\end{corollary}

	\begin{lemma}\label{lemma3}
		Let $D \in W_n, \ n\geq2$ be a simple derivation and  $T\in W_n$ be a derivation such that  $[T, D] = hD$ for a polynomial $h \in A$. Then either $T = \mu D$ for a polynomial $\mu \in A$ or there exist $D_3, \ldots, D_n \in W_n$ such that $D, T, D_3, \ldots, D_n$ form a basis of the free $A$-module $W_n$.
	\end{lemma} 
	
	\begin{proof}
		Write $D$ and $T$ in the coordinate form: 
		$$D = \sum_{i=1}^{n} a_i \, \frac{\partial}{\partial x_i}, \  T = \sum_{i=1}^{n} b_i \, \frac{\partial}{\partial x_i}, \ a_i, b_i \in A.$$
		Consider the $2 \times n$-matrix 
		$$
		G = \begin{pmatrix}
			a_1 & \dots & a_n \\
			b_1 & \dots & b_n 
		\end{pmatrix}.
		$$
		Denote by $\Delta$ the ideal of the polynomial ring $A$ generated by all minors of order $2$ of the matrix $G$. Since $[T, D] = hD$ it follows from Lemma \ref{invariant} that $D(\Delta) \subseteq \Delta$. By our assumption $D$ is a simple derivation, so we have $\Delta = 0$ or $\Delta =A$. In case $\Delta =0$,  we obtain, taking into account the equality $\gcd(a_1, \ldots, a_n) = 1,$ that $T = \mu D$ for some $\mu \in A$. Let $\Delta =A$. Then, by a generalization of the Quillen-Suslin theorem (see, for example, \cite{GQS}), there exist  derivations $D_3, \ldots, D_n \in W_n$, such that $D, T, D_3, \ldots, D_n$ is a basis of the free $A$-module $W_n$. The proof is complete.
	\end{proof}

	\begin{theorem}\label{th2}
		Let $D$ be a simple derivation in $W_{2}={\rm{Der}}(A),$ where  $ A=K[x, y]$. Then $AD$ is a maximal subalgebra of the Lie algebra $W_{2}$. Conversely, every maximal subalgebra of $\rm{rank}$ 1 of the Lie algebra $W_{2}$ is of  the form $M =AD$ for some simple derivation $D$ on the polynomial ring $A$.
	\end{theorem} 
	
	\begin{proof}
		Let $D = P \frac{\partial}{\partial x} + Q \frac{\partial}{\partial y}$ be a simple derivation. Assume, to the contrary, that $AD$ is not a maximal subalgebra of the Lie algebra $W_{2}$. Then there exists a subalgebra $M$ of $W_{2}$ such that $AD \subsetneq M \subsetneq W_{2}$. By Lemma \ref{lemma2}, the subalgebra $M$ is of $\operatorname{rank}$ $2$. We first show   that  $M$ can be chosen to be  a polynomial subalgebra. Choose  an arbitrary  element $T \in M \setminus AD$ such that $D$ and $T$ are linearly independent over $A.$ Such an element  $T$  exists since $\rm{rk}(M) = 2$. Let us show that $D$ and $[T, D]$ are linearly independent over $A$. Assume, to the contrary, that this is not the case. Then $[T, D] = hD$ for a polynomial $h \in A$ (recall that $D$ is a simple derivation, so the polynomials $P$ and $Q$ are coprime).
		
		By Lemma \ref{lemma3}, the elements $D, T$ form a basis of the $A$-module $W_{2}$. Write $T$ in coordinate form $T = R\frac{\partial}{\partial x} + S\frac{\partial}{\partial y}$ for some polynomials $R, S\in A.$ Then 
		$$
		\Delta :=  \begin{vmatrix}
			P & Q \\ R & S
		\end{vmatrix} \in \mathbb{K}^*.$$
		Since $[T, D] = hD,$  Corollary \ref{delta} implies  that  $D(\Delta) = (\mathrm{div}D)\Delta = 0.$ 
		Thus $\mathrm{div}D = 0$. It follows that $D$ is a Jacobi derivation, $D=D_{f}$ for a nonconstant polynomial  $f \in A.$ But then the  principal ideal  $(f)$ of the ring $A$  is obviously invariant under the  action of $D,$ which	contradicts the assumptions of the theorem.		Therefore $D$ and $[T, D]$ are linearly independent over the ring $A$.
		
		Now our aim is to show that the polynomial subalgebra $A\langle D, [T, D] \rangle$ generated by $T, [T, D]$ is a  subalgebra of  $M$.		
		Choose any polynomial $g \in A$. Then
		\[
		g[T, D] = [T, gD] - T(g)D \in M
		\]
		since $gD \in M$ and $T(g)D \in M$.
		
		Therefore $A[T, D] \subseteq M$. Taking into account the inclusion $AD \subseteq M$ and Lemma \ref{lemma1}, we see that $A\langle D, [T, D] \rangle$ is a polynomial subalgebra of $M$. Without loss of generality we may assume that $M$ is a polynomial subalgebra of $W_2$. Since $D, T \in M$ are linearly independent over $A$ we see that
		$$
		\Delta := \begin{vmatrix}
			P & Q \\ R & S
		\end{vmatrix} \neq 0 .$$
		Solving the linear system 
		$$
		D = P\frac{\partial}{\partial x} + Q\frac{\partial}{\partial y}, \ 		
		T = R\frac{\partial}{\partial x} + S\frac{\partial}{\partial y}
		$$
		for $ \frac{\partial}{\partial x} , \frac{\partial}{\partial y} $   we obtain the equalities 
		$$\Delta \frac{\partial}{\partial x} = SD - QT, \ \Delta \frac{\partial}{\partial y} = -R D + PT.
		$$
		Recall that  $M$ is a polynomial subalgebra of $W_2$ and $D, T\in M$, so 
		$$\Delta \frac{\partial}{\partial x} \in M, \ \Delta \frac{\partial}{\partial y} \in M.$$
		The latter means that $\Delta W_2 \subseteq M$. One can easily show that the set $$I = \{ f \in A \mid f W_2 \subseteq M \}$$
		is an ideal of the ring $A.$ Since $\Delta \in I$ the ideal $I$ is nonzero.  Take any element $D_1 \in W_2$.
		Then, for any $f \in I$, we have
		$$[D,fD_1] = D(f)D_1 + f[D,D_1].$$
		As $f[D,D_1] \in M$ and $[D,fD_1] \in M$ we have $D(f)D_1 \in M$. The latter means that 	$D(f)W_2 \subseteq M,$  i.e., $D(f)\in I,$	and therefore $D(I) \subseteq I.$		By conditions of the theorem, $D$ is a simple derivation and by the above proven  $I \neq 0$, so $I=A$. But then $W_2 =AW_2 \subseteq M$ and therefore  $M = W_2$. The obtained contradiction shows that $AD$ is a maximal subalgebra of $W_2$.
		
		Now let $M$ be a maximal subalgebra of rank $1$ from $W_2$. It follows from the inclusion $M \subseteq AM$ that $M =AM$, i.e., $M$ is a polynomial subalgebra. Then $M = ID$ for an ideal $I$ of the ring $A$ and a reduced derivation $D \in W_2$ (see for example, \cite{AMP}, Lemma 1).  Since $M$ is a maximal subalgebra of $W_2$ we have $I =A$, i.e., $M =AD$. Let us show that $D$ is a simple derivation of $A$. 		Take any proper ideal $J \subsetneq A$ which is invariant under the action of the derivation $D$. Then $(J W_2 + M)(J) \subseteq J$ which implies the inequality $JW_2+M \neq W_2$. As $M$ is a maximal subalgebra of $W_2$ we obtain the equality $JW_2+M = M$. The latter is impossible because $JW_2$ is a subalgebra of rank $2$ from $W_2$. The obtained contradiction shows that $D$ is a simple derivation on $A$. The proof is complete.
	\end{proof} 
	
	\begin{example}
		Take the derivation $D=\frac{\partial}{\partial x}+ (1+xy)\frac{\partial}{\partial y}\in W_2.$ It is known that $D$ is a simple derivation on $A=K[x, y]$ (see, for example \cite{Now}). Then, by Theorem \ref{th2}, the Lie algebra $M=A(\frac{\partial}{\partial x}+(1+xy)\frac{\partial}{\partial y}) $ is a maximal subalgebra of the Lie algebra $W_2.$
		
	\end{example}

	\section{On  maximal subalgebras of rank $n$ of $W_n$}
	In this section,  we consider maximal subalgebras of $W_n$ of maximal rank. 
	Note that every finite-dimensional maximal subalgebra $M$  of $W_n$ is of rank $n.$ Indeed,  suppose, to the contrary, that $\textrm rk M=k<n.$ Then $AM$ is of rank $k$ and $\dim AM=\infty .$ Therefore $M\subsetneq AM$. Of course, $AM\not = W_n$ since these algebras have distinct ranks over $A.$ The latter contradicts maximality of $M$ and obtained contradiction shows that $ \textrm rk M=n.$ 
	We also  need   some results concerning the following  subalgebras of $W_n$, which were  studied	in \cite{Makedonskyi}.
	\begin{definition} (\cite{Makedonskyi}).
		A Lie subalgebra $L \subseteq W_n({K})$ with $\dim_K L = n$ is called a basic subalgebra of $W_n$  if the $A$-module $AL$ generated by $L$ coincides with $W_n,$ i.e.,  $AL=W_n(K).$
		%		\annotation{This is an interpretation of the fragmented definition in the draft,} \annotation{required to make sense of the proof.}
	\end{definition}
	
	\begin{proposition}[Corollary 1, \cite{Makedonskyi}] \label{prop:div_free}
		If $L$ is a nilpotent or semisimple basic subalgebra of the Lie algebra  $W_n(K)$, then $L$ consists of divergence-free derivations.
	\end{proposition}
	
	\begin{lemma} \label{lm:max}
		Let $M$ be a finite-dimensional maximal subalgebra of $W_n(K),$ where $ n\geq 2$.
		Then every  non-zero ideal of the Lie algebra $M$ has rank $n$.
		In particular, $M$ is not  solvable.
	\end{lemma}
	\begin{proof}
		Suppose  that $I$ is an ideal of $M$ and  that ${\rm rk}(I)=k < n$. 
		Consider the subalgebra $M_I = M+AI$. Then 
		$M \subsetneq M_I$ because $M_I $ is an infinite dimensional subalgebra of $W_n.$  Since $M$ is a maximal subalgebra of $W_n$ we obtain $M_I=W_n.$ But $[M,  AI]\subseteq AI$, so $AI$ is an ideal of $W_n=M_I.$ This equality implies $AI=W_n$ which is impossible because ${\rm rk}(AI)=k<n.$ The obtained contradiction shows that ${\rm rk}(I)=n.$
		If there existed a finite-dimensional solvable maximal Lie subalgebra, then by Lie theorem it would have a one-dimensional ideal. 	By  the above mentioned this is impossible. Therefore every finite-dimensional maximal Lie algebra is not solvable.
	\end{proof}
	\begin{theorem}\label{th3}
		Let $M$ be a  finite dimensional maximal subalgebra of the Lie algebra $W_n.$ Then  $\dim M \geq n+1$. 
	\end{theorem}
	\begin{proof}
		First assume $\dim M = k < n$. Then ${\rm rk}M \le k < n$. Consider the $A$-submodule (which is also a subalgebra) $ AM=K[x_1, \dots, x_n]M$. This subalgebra is strictly contained in $W_n$, since ${\rm rk}(AM) = {\rm rk}M < n$. Furthermore, $AM$ strictly contains $M$, because $\dim M < \infty$ while $AM$ is infinite-dimensional (as $M \ne 0$). Thus, we have the chain $M \subsetneq AM \subsetneq W_n$, which means $M$ is not maximal. The latter contradicts the conditions of the lemma, therefore $\dim M = k \geq  n$
		
		Now, let $\dim M = n$. If $AM \ne W_n$, then, similarly to the case $\dim M < n$, $M$ is not maximal, as $M \subsetneq AM \subsetneq W_n$.
		Therefore,   $AM= W_n, $ that is,     $M$ is a basic subalgebra of the Lie algebra $W_n.$
		Since $\dim M = n$ the subalgebra $M$ is simple by Lemma \ref{lm:max}.
		Denote by $L$ the subalgebra of $W_n$ consisting of all divergence-free derivations. By  Proposition \ref{prop:div_free}, we have $M\subseteq L.$ But $L$ is strictly contained in the subalgebra $L_1$ consisting of all derivations with constant divergences. The latter means that $M$ is not maximal in $W_n$ which contradicts our assumption. Therefore $\dim M\geq n+1.$ 
		
	\end{proof}
	
	\begin{theorem}\label{th3a}
		Let $M$ be a polynomial maximal  subalgebra of $\operatorname{rank}$ $n$ from $W_n$. Then 
		
		(1)	$ 	M = \{ D \in W_n \mid D(I) \subseteq I \} $
		for an ideal $I$ of the ring $A$.
		
		(2) The subalgebra $I^2W_n$ is a nonzero proper ideal of $M$, so $M$ is a non-simple Lie algebra. 
	\end{theorem}
	\begin{proof}
		(1) Consider the ideal 		$	I = \langle f \in A \mid fW_n \subseteq M \rangle.	$
		Since $\operatorname{rk} M = n$  the quotient module $W_n/M$ has rank $0$. Using standard facts   about such modules over notherian rings (see, for example, \cite{Matsumura}) we get that the ideal $I$ is nonzero. Take any element $D \in M$ and show that $D(I) \subseteq I$. Really, for any $f \in I$ and any $D_1 \in W_n$ we have
		$$
		D(f)D_1 = [D, fD_1] - f[D, D_1].
		$$
		The both products from the right side of this equality belong to $M$, so $D(f)D_1 \in M$. The latter means that $D(f) \in I$, i.e., $D(I) \subseteq I$ for any $D \in M$. Since $M$ is a maximal subalgebra of $W_n$ we obtain
		$ M = \langle D \in W_n \mid D(I) \subseteq I \rangle .$
		
		(2) One can easily show that $I^2W_n$ is a nonzero ideal of the Lie algebra $M.$ If $I^2W_n=IW_n $ then $I^2=I.$ Applying the Nakayama lemma (see, for example, \cite{Matsumura}) we obtain that $I=0.$ The latter contradicts the above proven inequality $I\not =0.$ Therefore $I^2W_n$ is a proper nonzero ideal of the Lie algebra $M.$ 
	\end{proof}
	\section{Polynomial  maximal subalgebras of  $W_n(K)$}
	\begin{lemma}[see, for example,  \cite{Now}] \label{lm:Euler}
		Let $E_n = \sum_{i=1}^n x_i \frac{\partial}{\partial x_i}$ be the Euler derivation,
		$D \in W^{[i]}_n$ and let $f \in A$ be a homogeneous polynomial of degree $m$.
		Then
		
		(1) $E_n(f) = mf;$
		
		(2) $ [E_n, D] = iD,$
		in particular, for $D \in W^{[0]}_n$ it holds $[E_n, D] = 0;$
		
		(3)  ${\rm div}(f E_n) = (m+n)f.$
		
	\end{lemma}

	\begin{lemma} \label{lm:graded} (see, for example, \cite{CP25}).
		Let $L$ be a subalgebra of $W_n({K})$ such that $E_n \in L$ (where $E_n$ is the Euler derivation). Then $L$ is a graded subalgebra of $W_n({K})$ with the grading induced from the standard grading of $W_n({K})$.
	\end{lemma}
	
	\begin{remark}\label{2modules}
		
		If $U, V$ are irreducible non-isomorphic $W_n^{[0]}$-modules, then the module $U\oplus V$ has only three proper submodules:   $U, V$ and $0$.
	\end{remark}
	The  next result  was first stated in \cite{KS} (Chapt.1, p.267)  for Lie algebras over fields  of positive characteristic, but it is true  for characteristic zero. An explicit proof of this result can  be found in \cite{CP25}.
	\begin{theorem}\label{ThKS}\cite{KS}, (Chapt.1, p.267).
		Let the Lie algebra $W_n(K) = W_n, n\geq 2$ be written as a direct sum of homogeneous 
		components of the standard grading
		\begin{equation} \label{eq:e6}
			W_n = W_n^{[-1]} \oplus W_n^{[0]} \oplus \dots \oplus W_n^{[m]} \oplus \dots.
		\end{equation}
		Then $L = W_n^{[0]}$ is a subalgebra of \  $W_n$, $L \simeq \mathfrak{gl}_n(K)$
		and every summand of the sum \eqref{eq:e6} is a finite dimensional module over $L$.
		Every $L$-module $W_n^{[m]}, m \ge 0$ is a direct sum $W_n^{[m]} = M_m \oplus N_m$
		of two irreducible submodules, where $M_m$ consists of divergence-free derivations and $N_m$
		consists of all the derivations from $W_n^{[m]}$ that are polynomial multiple of the Euler derivation $E_n$.
	\end{theorem}	
	
	\begin{proposition}\label{products} (see, \cite{CP25})
		Let $W_n = W_{n}^{[-1]}\oplus W_{n}^{[0]}\oplus\dots\oplus W_{n}^{[i]}\oplus\dots$ be the decomposition of $ W_n $ into  direct sum of homogeneous components of the standard grading on $W_{n}, n\geq 2.$ Let $W_{n}^{[i]}=M_{i}\oplus N_{i}$, $i\ge 0$, where $M_{i}$, $N_{i}$ are irreducible submodules of the $W_{n}^{[0]}$-module $W_{n}^{[i]}$ from Theorem \ref{ThKS}. Then
		\begin{enumerate}
			\item If  $S, T$ are submodules of the $W_n^{[0]}$-module $W_n,$ then $[S,  T] $ is also a submodule of $W_n;$
			\item $[M_{i}, M_{j}] = M_{i+j}$ for all $i,j\ge 0$;
			\item $[N_{i}, N_{j}] = N_{i+j}$ for all $i,j\ge 0$, $i\ne j$; $[N_{i}, N_{i}] = 0$, for all $i\ge 0$;
			\item $[M_{i}, N_{j}] = W_{n}^{[i+j]}$ for all $i,j\ge 0$, except the cases:
			(a) $i=j=0$, where $[M_{0}, N_{0}] = 0$, 
			
			(b) $i= 0$, $j= 1,2,\dots$, where $[M_{0}, N_{j}] = N_{j}$,
			
			(c) $ i\geq 1 , j=0,  $ where $ [M_i, N_0]=M_i.$
			\item $[W_n^{[-1]}, N_{j}] = W_{n}^{[j-1]}$ for all $j\ge 0$
		\end{enumerate}
	\end{proposition}
	
	\begin{corollary}\label{W_iWj} (see \cite{CP25})
		Let  $W^{[i]}_n$ and $W^{[j]}_n$ be homogeneous components  of the standard grading $W_n= \underset{i\ge-1}{\bigoplus}W_{n}^{[i]}, n\geq 2.$ Then  the next equalities are satisfied: 
		$$[W^{[i]}_n, W^{[j]}_n]= W^{[i+j]}_n, i, j\geq 0,$$ except the case when  $i=j=0$, where $[W_{n}^{[0]}, W_{n}^{[0]}]=M_{0}\simeq \mathfrak{sl}_{n}(K)$. Obviously $[W_n^{[-1]}, W_n^{[-1]}]=0$ since $[\frac{\partial}{\partial x_i}, \frac{\partial}{\partial x_j}]=0, i,j=1,\ldots ,n. $
	\end{corollary}
	\begin{proof}
		In view of Proposition \ref{2modules} and Proposition \ref{products} we need only to show that 
		$[W_{n}^{[-1]}, W_{n}^{[m]}]\not = 0$, $m\ge 0$. But this inequality can be checked immediately. 
	\end{proof}
	\begin{theorem}\label{th5}
		The subspace $L = W_n^{[-1]} \oplus W_n^{[0]} \oplus N_1$  is a finite dimensional maximal subalgebra of the Lie algebra $W_n$. And vice versa, every finite dimensional maximal subalgebra of $W_n$ containing $W_n^{[0]}$ coincides with $L$.
	\end{theorem}
	\begin{proof}
		By Proposition  \ref{products}, 
		$$W_n^{[0]} = M_0 \oplus N_0,  \ W_n^{[1]} = M_1 \oplus N_1,$$ and the inclusions $[M_0, N_1] \subseteq N_1$, $[N_0, N_1] \subseteq N_1$ hold. So, $L$ is a subalgebra of the Lie algebra $W_n(\mathbb{K})$.
		Let us prove that $L$ is a maximal subalgebra of $W_n$. Take any element $D \in W_n \setminus L$ and denote by $S = \langle L, D \rangle$ the subalgebra of $W_n$ generated by $L$ and $D$. Since $L$ contains the Euler derivation $E_n$, the algebra $S$ is graded (by Lemma \ref{lm:graded}) with the grading induced from $W_n$, $S = \bigoplus_{i \ge -1} S_i$, $S_i = S \cap W_n^{[i]}$. Therefore $D = D_{-1} + D_0 + D_1 + \dots + D_k$ for some $D_i \in S_i$. The element $D$ does not belong to $L$, so there exists $i_0$ such that $D_{i_0} \in S \setminus L$. Thus, without loss of generality, one can assume that $D \in S_{i_0}$ for some $i_0 \ge 1$, moreover we can choose $i_0$ to be minimal possible.
		First, let $i_0 = 1$, i.e., $D = D_1 \in S_1 \setminus L$, $S_1 \subseteq W_n^{[1]}$. Therefore, $S_1$ is a nonzero $W_n^{[0]}$-submodule of $W_n^{[1]}$ (recall that $W_n^{[0]} \subseteq S$).
		By Remark \ref{2modules}, either $S_1=W_n^{[1]}$ or $S_1 = M_1$, or $S_1 = N_1$. But the latter equality is impossible because $D_1 \in S_1 \setminus L$. If $S_1 \supset M_1$ then $[M_1, N_1] \subseteq S_2$ and $[M_1, N_1] = W_n^{[2]}$ by Proposition \ref{products}. But then, by Corollary \ref{W_iWj}, 
		$W_n^{[1]} = [W_n^{[-1]}, W_n^{[2]}] \subset S$, and therefore 
		$W_n^{[i]} \subset S$ for $i \ge 1$. The latter means $S = W_n$.
		If $i_0 > 1$, then $W_n^{[1]} \subset S$, and again
		$S = W_n$.
		
		Next, let us prove the second part of the statement of the theorem.
		Let $T \subset W_n$ be a finite-dimensional maximal subalgebra  that contains $W_n^{[0]}$. Since  $E_n\in W_n^{[0]}\subseteq T$, where $E_n$ is the Euler derivation, the subalgebra $T$ is a graded subalgebra by Lemma \ref{lm:graded}. Then  $T=\bigoplus_{i \ge -1}T_i,$ 		where $T_i=T\cap  W_n^{[i]}$ are the homogeneous components, $T_0 = W_n^{[0]}.$  If $T_{-1}=0,$ then the subalgebra $T$ lies in the infinite dimensional proper subalgebra $\bigoplus_{i \ge 0}W_n^{[i]}$ of the Lie algebra $W_n.$ The latter contradicts the maximality of $T,$ so we obtain $T_{-1}\not =0.$ But then $T_{-1}=W_n^{[-1]}$ since $W_n^{[-1]} $ is an irreducible module over $W_n^{[0]}.$ Thus, $W_n^{[-1]}\oplus W_n^{[0]}\subseteq T.$ The subalgebra $W_n^{[-1]}\oplus W_n^{[0]}$ is not maximal in $W_n$ in view of the part (1) of the theorem, therefore $T\not =W_n^{[-1]}\oplus W_n^{[0]}$. So, we have  $T_i\not =0$ for some $i\geq 1.$ First consider the case when $T_i\not =0$ for $i\geq 2.$ 
		As $T_i$ is a $	W_n^{[0]}$-submodule of $W_n^{[i]}$ we have by Remark \ref{2modules} that $T_i\in \{ M_i, N_i, W_n^{[i]}\} .$	 If $T_i=M_i,$ then $[M_i, M_i]=M_{2i} \subset  T.$ Analogously one can show that $M_{ki}\subset T$ for $k\geq 2. $ The latter contradicts our assumption that  $T$ is finite dimensional and therefore $T_i\not =M_i.$ Analogously one can show that $T_i\not = W_n^{[i]}.$ It follows from these considerations that $T_i=N_i.$ But then $W_n^{[i-1]}=[W_n^{[-1]}, N_i]\subseteq T.$ Applying  Proposition \ref{products} we get $[W_n^{[i-1]}, N_i]=W_n^{[2i-1]}\subseteq T.$ But then $W_n^{[k(2i-1)]}\subseteq T, k\geq 1 $ and $T$ is infinite dimensional. This contradicts our assumption about $T.$ The obtained contradiction shows that $T_i=0$ for $i\geq 2,$ i.e., $$T\subseteq W_n^{[-1]} \oplus W_n^{[0]} \oplus W_n^{[1]}.$$
		Repeating the above considerations one can show that $T_1\not =M_1$ and $T_1\not =W_n^{[1]}. $ Therefore $T_1=N_1$ and $T=W_n^{[-1]} \oplus W_n^{[0]} \oplus N_1.$
	\end{proof}
	
	\begin{proposition}
		The subalgebra $L = W_n^{[-1]} \oplus W_n^{[0]} \oplus N_1$ from the statement of the previous theorem
		is isomorphic to the special linear Lie algebra $\mathfrak{sl}_{n+1}(\mathbb K)$.
	\end{proposition}
	\begin{proof}
		Define the linear map $\varphi \colon L \to \mathfrak{sl}_{n+1}(\mathbb K)$ by:
		$$		\varphi(\partial_{i})= E_{n+1,i}, \ 
		\varphi(x_{i}\partial_{j})= E_{ij} \quad (i \neq j), \ 
		\varphi(x_{i}\partial_{i})= E_{ii} - \frac{1}{n+1} I_{n+1}, \
		\varphi(x_{i}E)= -E_{i,n+1}.
		$$	
		One can  easily prove that  $\varphi$ is a Lie algebra homomorphism. The verification involves checking that $\varphi$ preserves the Lie bracket for all pairs of basis elements. Key cases include:
		\begin{itemize}
			\item $[\partial_{i},\partial_{j}]=0$ maps to $[E_{n+1,i},E_{n+1,j}]=0$.
			\item $[\partial_{i},x_{j}\partial_{k}]=\delta_{ij}\partial_{k}$ maps to $[E_{n+1,i},E_{jk}]=\delta_{ij}E_{n+1,k}$ when $j\neq k$, with similar handling for $j=k$.
			\item $[\partial_{i},x_{j}E]=\delta_{ij}E+x_{j}\partial_{i}$ maps to $[E_{n+1,i},-E_{j,n+1}]=E_{j,i}-\delta_{ij}E_{n+1,n+1}$, matching $\varphi(\delta_{ij}E+x_{j}\partial_{i})$.
			\item $[x_{i}\partial_{j},x_{k}\partial_{t}]=\delta_{jk}x_{i}\partial_{t}-\delta_{ti}x_{k}\partial_{j}$ maps to $[E_{ij},E_{kl}]=\delta_{jk}E_{ti}-\delta_{ti}E_{kj}$ after careful case analysis for diagonal elements.
			\item $[x_{i}\partial_{j},x_{k}E]=\delta_{jk}x_{i}E$ maps to $[E_{ij},-E_{k,n+1}]=-\delta_{jk}E_{i,n+1}$.
			\item $[x_{i}E,x_{j}E]=0$ maps to $[-E_{i,n+1},-E_{j,n+1}]=0$.
		\end{itemize}
		All checks confirm that $\varphi([u,v])=[\varphi(u),\varphi(v)]$ for basis elements, and $\varphi$ is bijective. Hence, $\varphi$ is an isomorphism of Lie algebras.
	\end{proof}
	
	\begin{proposition}
		Every infinite-dimensional maximal subalgebra of $W_n$ that contains $W_n^{[0]}$ coincides with the subalgebra
		$\bigoplus_{k \ge 0} W_n^{[k]}$.
	\end{proposition}
	\begin{proof}
		Let $L$ be an infinite-dimensional maximal subalgebra of $W_n$. 
		As in the previous theorem $L$ is a graded subalgebra
		$	L = \bigoplus_{i \ge -1} L_i,	$
		where $L_i=L_i\cap W_n^{[i]}$ are the homogeneous components,
		$L_i \in \{0, M_i, N_i, W_n^{[i]}\}$ for $i \ge 0$, and
		$L_{-1}$ is either zero or coincides with $W_n^{[-1]}$.
		It suffices to prove that $L_{-1} = 0$.
		Suppose that $L_{-1} = W_n^{[-1]}$.
		$L$ is maximal subalgebra and it does not lie in the subalgebra of all derivations with constant divergence,
		hence $L_k \supset N_k$ for some $k \ge 1$. Then, $k > 1$, since $W_n^{[-1]} \oplus W_n^{[0]} \oplus N_1$
		is a maximal subalgebra by the previous theorem.
		Then, $[W_n^{[-1]}, N_{k}] = W_{n}^{[k-1]} \subset L_{k-1}$ by Proposition  \ref{products}.
		Further, $L_i = W_n^{[i]}$ for $i < k$.
		Then 
		$$L_k \supset [L_1, L_{k-1}] = [W_n^{[1]}, W_n^{[k-1]}] = W_n^{[k]}.$$
		Applying the same argument we obtain that $L_i = W_n^{[i]}$ for $i \ge k$, and thus $L = W_n$. The latter contradicts our assumption on $L.$  The obtained contradiction shows that $L=\bigoplus_{k \ge 0} W_n^{[k]}.$
		
	\end{proof}
	
	\begin{theorem}
		Let $I_k = (x_{k+1}, \dots, x_n)$ be the ideal in $A$ generated by the last $n-k$ variables. Denote by $L_k$ the subalgebra that preserves the ideal $I_k$, i.e., 
		$L_k = \{ D \in W_n \mid D(I_k) \subseteq I_k \}.$
		Then $L_k$ is a maximal subalgebra of $W_n$ for $i=1, \ldots , n-1.$ 
		
	\end{theorem}
	
	\begin{proof} One can easily to show that 
		$$L_k=A\frac{\partial}{\partial x_1} + \dots +A\frac{\partial}{\partial x_k} + I_k \frac{\partial}{\partial x_{k+1}} + \dots + I_k \frac{\partial}{\partial x_n}.
		$$
		Since $L_k$ is a Lie subalgebra in $W_n$ we can consider $L_k$  and $W_n$ as   $L_k$-modules.  
		Consider the quotient module $Q=W_n/L_k.$ Denote ${\frac{\overline{\partial}}{\partial x_i}}=\frac{\partial}{\partial x_i}+L_k, i=1, \ldots , n.$ One can easily  show that 
		$$			Q = W_n / L_k=(A/I_k){\frac{\overline{\partial}}{\partial x_{k+1}}} + \dots + (A/I_k){\frac{\overline{\partial}}{\partial x_n}}.$$
		Note that the quotient ring $A/I_k $ is isomorphic to the polynomial ring $K[x_1, \dots, x_k],$ so we have  
		$$W_n/L_k\simeq {K}[x_1, \dots, x_k]{\frac{\overline{\partial}}{\partial x_{k+1}}}+\cdots + {K}[x_1, \dots, x_k]{\frac{\overline{\partial}}{\partial x_n}}.
		$$
		Our aim is to show that $Q=W_n/L_k$ is an irreducible $L_k$-module. 
		Take any $\overline{D}\in Q.$	Then $\overline{D}$ is of the form  
		$$		
		\overline{D} = f_{k+1}(x_1, \dots, x_k){\frac{\overline{\partial}}{\partial x_{k+1}}} + \cdots + f_n(x_1, \dots, x_k){\frac{\overline{\partial}}{\partial x_n}}.
		$$	
		Since  $\frac{\partial}{\partial x_i} \in L_k$ for $1 \le i \le k$, we have:
		$$[\frac{\partial}{\partial x_i},  \overline{D}] = \frac{\partial f_{k+1}}{\partial x_i} \frac{\overline{\partial}}{\partial x_{k+1}} + \dots + \frac{\partial f_n}{\partial x_i} \frac{\overline{\partial}}{\partial x_n} \in L_k\langle \overline{D}\rangle .
		$$
		$$	\overline{D}_1 = c_{k+1}\frac{\overline{\partial}}{\partial x_{k+1}} + \dots + c_n \frac{\overline{\partial}}{\partial x_n} \in L_k\langle \overline{D}\rangle ,
		$$
		where $c_{k+1}, \dots, c_n \in K$ and $c_j \ne 0$ for some $j,$ $k+1\leq j\leq n.$
		
		Then $[-x_j \frac{\partial}{\partial x_j}, \overline{D}_1] = c_j \frac{\overline{\partial}}{\partial x_j} \in L_k \langle \overline{D} \rangle$.
		Finally, 
		\begin{equation*}
			[-x_j f(x_1, \dots, x_k) \frac{\partial}{\partial x_q}, \frac{\overline{\partial}}{\partial x_j}] = f(x_1, \dots , x_k) \frac{\partial}{\partial x_q} \in L_k \langle \overline{D} \rangle,
		\end{equation*} 
		for all $q = k+1,\dots,n$, $f \in \mathbb{K}[x_1, \dots, x_k]$.
		
		Thus, $L_k \langle \overline{D} \rangle = Q$ for an arbitrary nonzero $\overline{D} \in Q, $ which means $Q$ is an irreducible $L_k$-module.
		It follows directly that $L_k$ is a maximal subalgebra in $W_n$. 
		Indeed, if there existed an intermediate subalgebra $L_k \subsetneq M \subsetneq W_n$, then there would exist a non-trivial $L_k$-submodule $M/L_k$ in $W_n/L_k = Q$, which contradicts the irreducibility of the module $Q$.
	\end{proof}
	
	\begin{remark}
		$L_k$ can be interpreted as the Lie algebra of polynomial vector fields that leave the coordinate $k$-plane $x_{k+1} = \dots = x_n = 0$ invariant.
	\end{remark}
	
	\section*{Acknowledgments}	%
	{The first author of this work was partially supported by a grant
		from the Simons Foundation (SFI-PD-Ukraine-00014586, Y.Y.C.).
		Yevhenii Chapovskyi expresses his gratitude for the financial support of the National Academy of Sciences of Ukraine
		within the framework of the project 0125U002856 for young scientists.}

\end{document}